\documentclass[12pt]{article}
\usepackage{amsmath,amssymb,dsfont,tikz-cd,amsthm,inputenc}

\title{A Categorical Notion of Precompact Expansions}
\author{Keegan Dasilva Barbosa}
\date{February 26, 2020}

\renewcommand{\AA}{\mathbf{A}}
\newcommand{\BB}{\mathbf{B}}
\newcommand{\CC}{\mathbf{C}}
\newcommand{\GG}{\mathbf{G}}

\newcommand{\FF}{\mathcal{F}}

\newcommand{\NN}{\mathbb{N}}

\newcommand{\QQ}{\mathbb{Q}}
\newcommand{\SSSS}{\mathbf{S}}

\newcommand{\KK}{\mathcal{K}}
\newcommand{\XX}{\mathbf{X}}
\newcommand{\YY}{\mathbf{Y}}
\newcommand{\ZZ}{\mathbb{Z}}
\newcommand{\PP}{\mathbb{P}}

\newcommand{\Gra}{\mathbf{Gra}}
\newcommand{\LOGra}{\mathbf{LOGra}}

\newtheorem{theorem}{Theorem}[section]
\newtheorem{prop}{Proposition}
\newtheorem{corollary}{Corollary}[theorem]
\newtheorem{lemma}[theorem]{Lemma}
\theoremstyle{definition}
\newtheorem{definition}{Definition}[section]

\DeclareMathOperator{\Age}{Age}
\DeclareMathOperator{\Aut}{Aut}

\begin{document}
\maketitle
\begin{abstract}
	We generalize the notion of relational precompact expansions of Fra\"iss\'e classes via functorial means, inspired by the technique outlined by Laflamme, Nguyen Van Th\'e and Sauer in their paper \textit{Partition properties of the dense local order and a colored version of Milliken's theorem}. We also generalize the expansion property and prove that categorical precompact expansions grant upper bounds for Ramsey degrees. Moreover, we show under strict conditions, we can also compute big Ramsey degrees. We also apply our methodology to calculate the big and little Ramsey degrees of the objects in $\Age(\SSSS(n))$ for all $n\geq 2$. 
\end{abstract}
\section{Introduction}
In recent years, it has become quite apparent that topological dynamics and Ramsey theory are inseparable from one another. A great example of this was the KPT correspondence in \cite{KPT}, where it was proven that a Fra\"iss\'e class of rigid structures $\KK$ has the Ramsey property if and only if the group $G=\Aut(\FF)$ was extremely amenable, where $\FF$ is the corresponding Fra\"iss\'e structure. Moreover, the paper outlined a way to compute the universal minimal flow explicitly, given the class had the order property. This was then elaborated on further in \cite{PrecompExp}, where it was shown that precompact relational expansions could also be used to compute the universal minimal flow. Moreover, it was shown in \cite{MetUniv} that precompact relational expansions are optimal, in that whenever $M(G)$ is metrizable, a precompact relational expansion of rigid structures with the Ramsey Property exists. The main drawback to precompact relational expansions however, is that of language. Given a class $\KK$, one needs to add relations $R_i$ to its language $L$. However, which relations to add may not be clear, as the structures themselves may be quite complex. Category theory on the other hand, has proven to be quite useful at proving Ramsey theoretic results by relating structures, despite possibly having different languages. One can see instances of this in \cite{Kubis} and \cite{BigRamsey}. This begs the questions, is there a reasonable notion of precompact relational expansion in the categorical setting? Could such a notion provide a direct link between category and topological dynamics? We assert a positive answer to both these questions. Moreover, we prove the following result. 
\begin{theorem}
	Let $\KK$ be a Fra\"iss\'e class with Fra\"iss\'e limit $\FF$ and $G= \Aut(\FF) $. The following are equivalent. 
	\begin{enumerate}
		\item $\KK$ admits a categorical precompact expansion $\KK^*$ with finite Ramsey degrees. 
		\item Every object in $\KK$ has a finite Ramsey degree.
		\item $M(G)$ is metrizable. 
	\end{enumerate}
\end{theorem}
The equivalence of $2$ and $3$ has become a well known result in the study of structural Ramsey Theory, with proofs appearing in \cite{MetUniv} and \cite{Zucker}. Our main goal will be to prove both $1\Rightarrow 2$ and $1\Rightarrow 3$ to show the interconnectedness of category theory with combinatorics and topological dynamics, without trying to rely too heavily on the equivalence $2\iff 3$. By doing so, we hope to further establish the significance of category theory in the field and show it too is a major central component to structural Ramsey theory. Finally, we show an application. In particular, we prove the following theorem. 
\begin{theorem}
	For any $n\geq 2$, the class $\Age(\SSSS(n)) $ has finite Ramsey degrees. Moreover, for any $\AA \in \Age(\SSSS(n))$ 
	\begin{itemize}
		\item $t(\AA) = \frac{n|A|}{|\Aut(\AA) |} $. 
		\item $T(\AA) = t(\AA) \tan^{(2|A|-1)} (0) $ 
	\end{itemize}
\end{theorem}
It was shown in \cite{ColouredMilliken} that this is true for $n=2$. The methods we will use to prove this result are reminiscent of \cite{ColouredMilliken}, where the authors arguably used a categorical precompact expansion, without explicitly defining such machinery. 
\\
\\
The paper will be split in to 4 core components. Section 2 will serve as preliminaries for the field of structural Ramsey theory. It serves to showcase all the relevant theorems and definitions needed to understand the remaining sections. Section 3 will be where we define a categorical precompact expansion. It will also be where we link category theory to topological dynamics and prove $1 \Rightarrow 3$ in relation to theorem 1.1. Section 4 will be where we show how to use a categorical precompact expansion to extract bounds on Ramsey degrees. Consequently, this is where we will complete the proof of theorem 1.1. Finally, section 5 will be where we prove theorem 1.2.    
\section{Preliminaries}
This section will be split in to three parts, Fra\"iss\'e theory, Ramsey theory and category theory. They will serve as a brief introduction to the three fields and can be skipped if the reader is already comfortable with them.
\subsection{Category Theory}
We start with the standard definition of a category.
\begin{definition}
	A category $\mathcal{C}$ is a class of objects $\textbf{Ob}(\mathcal{C})$ and a class of morphisms between objects $\text{hom}(\mathcal{C})$ that satisfy
	\begin{itemize}
		\item $\forall \AA,\BB,\CC \in \textbf{Ob}(\mathcal{C})$, there is a binary relation $\circ: \text{hom}(\AA,\BB) \times \text{hom}(\BB,\CC) \rightarrow \text{hom}(\AA,\CC)$ called a composition.
		\item Composition is associative ($(f\circ g)\circ h = f\circ(g\circ h)$ ).
		\item $\forall \AA\in \text{Ob}(\mathcal{C})$, $\exists 1_{\AA} \in \text{hom}(\AA,\AA)$ such that $\forall \BB, \CC \in \text{Ob}(\mathcal{C})$ $\forall f\in \text{hom}(\AA,\BB)$ $\forall g\in \text{hom}(\CC,\AA)$, $1_{\AA} \circ g= g$ and $f\circ 1_{\AA} = f$. 
	\end{itemize} 
\end{definition}
We will be interested entirely in categories of finite and countable models over a fixed signature $\mathcal{L}$ which contains only countably many relation symbols. Models will always be expressed as boldfaced capital letters, with their underlying universe expressed as the unbolded letter eg. $\AA$ and $A$. All categories will have their morphisms $\text{hom}(\AA,\BB)$ defined to be the functions satisfying
\begin{align*}
f:A\rightarrow B\\
R_i^{\AA}(x_1,...,x_{a(i)}) &\iff R_i^{\BB}(f(x_1),...,f(x_{a(i)}))
\end{align*}
For finite structures, $\text{hom}(\AA,\AA) = \Aut(\AA)$, the collection of all automorphisms. Given a model $\BB$ and $A\subseteq B$, we define $\BB\upharpoonright A $ to be the model induced by $\BB$ restricted to $A$. Given two models over the same signature $\AA$ and $\BB$, we define 
\begin{align*}
\binom{\BB}{\AA} &= \{\BB\upharpoonright A^{'} : A^{'} \subseteq B, \; \BB\upharpoonright A^{'} \cong \AA  \}
\end{align*}
Note that $\binom{\BB}{\AA}$ can be realized as an equivalence class over hom-sets. Namely, the relation $\sim $ on $\text{hom}(\AA,\BB)$ defined by $f\sim h  \iff \exists g\in \text{Aut}(\AA), $ $f\circ g =h $ is an equivalence relation and
\begin{align*}
\text{hom}(\AA,\BB)/\sim &= \binom{\BB}{\AA}
\end{align*}
It will be necessary for us to relate different categories of structures, and to do so, we will need functors. Functors are the most natural way to relate categories to one another.
\begin{definition}
	Suppose $\mathcal{C}$ and $\mathcal{D}$ are categories. A functor is a map $F$ that sends objects from $\mathcal{C}$ to objects from $\mathcal{D}$, sends morphisms from $\text{hom}_{\mathcal{C}}(\AA,\BB)$ to morphisms from $\text{hom}(F(\AA),F(\BB)) $ in such a manner that $F(f\circ g) = F(f)\circ F(g) $ and $F(1_{\AA}) = 1_{F(\AA)}$.  
\end{definition}
\subsection{Fra\"iss\'e Theory}  
A central notion in Fra\"iss\'e Theory is that of the Age of a structure. An age of a structure is the category of all finite induced substructures. That is
\begin{align*}
	\Age(\BB) &= \{\AA: \binom{\BB}{\AA} \neq \emptyset, \;\; |A| \text{  finite.}     \}
\end{align*}
Interestingly enough, Ages are uniquely defined by categories that satisfy a type of upward and downward closure and witness a countable skeleton i.e one only needs countably many structures to define all structures up to isomorphism.
\begin{definition}
	Suppose $\KK$ is a class of finite structures over the same signature $\mathcal{L}$. We say that $\KK$ is an Age if and only if
	\begin{itemize}
		\item $\KK$ has the Heriditary Property (HP). That is, for any $\BB \in \KK$, if $\AA$ embeds in to $\BB$, then $\AA \in \KK$.
		\item $\KK$ has the Joint Embedding property (JEP). If $\AA, \BB \in \KK$, there is a $\CC \in \KK$ for which $\AA, \BB \in \Age(\CC)$. 
		\item There are only countably many nonisomorphic structures in $\KK$.
	\end{itemize}
\end{definition}
A consequence of this definition is that $\KK$ is an Age if and only if there is a countable structure $\FF$ for which $\KK = \Age(\FF)$. That is, there is a structure which is universal over $\KK$. A great example of such classes are the class of all finite graphs $\Gra$ and the class of all linearly ordered graphs $\LOGra$. Of course, the existence of a universal structure on its own, does not tell us much about an Age. For example, both the Random graph and a disjoint union of all complete graphs are universal over $\Gra$. But the Random graph is distinct from a countable union of complete graphs. Namely, the Random graph has a high level of categoricity which is a consequence of a property called ultrahomogeneity.
\begin{definition}
	We say a structure $\FF$ is ultrahomogeneous when for any $\AA \in \Age(\FF)$ and any pair $f,g \in \text{hom}(\AA,\FF)$, there is an automorphism $h\in \Aut(\FF)$ such that $h\circ f = g $. 
\end{definition}
This leads us to the main Theorem of this section, which is a classical result of Fra\"iss\'e. An age that satisfies an added upward closure condition.
\begin{definition}
	We say a class of structures satisfies the Amalgamation Property (AP) if $\forall \AA,\BB_1,\BB_2 \in \KK$ and embeddings $f_i:\AA\rightarrow \BB_i$, there is a $\CC \in \KK$ and embeddings $g_i : \BB_i \rightarrow \CC$ such that $g_1\circ f_1 = g_2\circ f_2$.
\end{definition}
\begin{theorem}
	\textbf{(Fra\"iss\'e)} Suppose $\KK$ is an Age with AP. There is a unique up to isomorphism, ultrahomogeneous structure $\FF$ such that $\KK = \Age(\FF)$. Similarly, if $\FF$ is ultrahomogeneous, then $\Age(\FF)$ satisfies AP. 
\end{theorem}
So in the case of the class $\Gra$, the Random graph is the unique ultrahomogeneous universal structure. Objects of this type are central to the modern study of structural Ramsey theory.
\subsection{Ramsey Theory}
The classical Ramsey theorem states that for any $n,m$ and $k$ in $\NN$, there is an $N\in \NN$ such that for any colouring $\chi: [N]^m \rightarrow k$, there is a $M\subseteq N$ of size $n$ for which $\chi \upharpoonright{[M]^m } $ is constant. We work with a generalization of this. That is,
\begin{definition}
	We say a class $\KK$ has the Ramsey Property (RP) if for all $\AA, \BB \in \KK$, $k\in \NN$, there is a $\CC \in \KK$ such that $\forall \chi : \binom{\CC}{\AA} \rightarrow k$ $\exists \BB^{'} \in \binom{\CC}{\BB}$ for which $ \chi \upharpoonright{\binom{\BB^{'}}{\AA}}$ is constant. We often write this with the Rado notation $\CC \rightarrow (\BB)_k^{\AA}$.  
\end{definition} 
Note, it is not always the case that a class has the Ramsey property. However, for any class $\KK$, every element has a Ramsey degree.
\begin{definition}
	Let $\KK$ be a class of finite structures. For any $\AA\ \in \KK$, we define the the Ramsey degree of $\AA$ to be $t_{\KK} (\AA) = \text{min}\{t \in \NN : \exists \CC \in \KK \; \CC\rightarrow (\BB)_{k,t}^{\AA}  \}  $. If the minima does not exist,we set $t_{\KK}(\AA) = \infty$. 
\end{definition}
In the above definition, the statement $\CC\rightarrow (\BB)_{k,t}^{\AA} $ is the exact same statement as $ \CC\rightarrow (\BB)_{k}^{\AA}$, except now we allow our colouring to take at most $t$ many values (opposed to being constant). In this sense, $\CC\rightarrow (\BB)_{k}^{\AA} $ is equivalent to $\CC\rightarrow (\BB)_{k,1}^{\AA}$. There is also an infinite dimensional version of Ramsey degree called a big Ramsey degree. 
\begin{definition}
	Let $\FF$ be a countably infinite structure and set $\KK = \text{Age}(\FF)$. For all $\AA\in \KK$, we declare the big Ramsey degree of $\AA$ to be $ T_{\KK}(\AA) =\text{min}\{ t: \FF \rightarrow (\FF)_{k,t}^{\AA}    \} $ or $\infty$ if the minima does not exist. 
\end{definition}
Note, the above definition is not a very well founded one. Namely, many universal structures can exist, so which one are we interested in the big Ramsey degrees of? For the purpose, of this paper our classes will be Fra\"iss\'e, and $\FF $ will denote the unique ultrahomogeneous structure. \\
\\
This brings us to KPT theory which introduces topological dynamics to the equation. Let us define $ S_\infty$ to be the permutation group of $\NN$ endowed with the topology of point-wise convergence. We have the following classification results
\begin{definition}
	A subgroup $G\subseteq S_\infty$ is closed if and only if it is isomorphic to $ \Aut(\FF)$ for a Fra\"iss\'e structure $\FF$. 
\end{definition}
It was also shown in [1] that RP for a Fra\"iss\'e class is equivalent to the extreme amenability of the automorphism group.
\begin{definition}
	Suppose $G$ is a topological group. We say that it is extremely amenable if whenever $G$ acts on a compact Hausdorff space $X$, it admits a fixed point.
\end{definition}
\begin{theorem}
	\textbf{(KPT Theorem)} Suppose $\FF$ is a Fra\"iss\'e structure. Let $\KK=\Age(\FF)$ and $G= \Aut(\FF)$. Then the following are equivalent. 
	\begin{itemize}
		\item $\KK$ is a class of rigid structures and has RP.
		\item $G$ is extremely amenable. 
	\end{itemize}
\end{theorem}
In [6], the above link between topological dynamics is used to reproduce the Ramsey property from a class by looking at an expansion.Let $\KK$ be a class of finite structures with signature $L$. Take $R_i$ to be a countable collection of relations indexed by some set $I$, with $R_i$ independent of $L$. We write $\vec{R}$ to mean the tuple $(R_i)_{i\in I}$. We define $L^*= L\cup \{R_i : i\in I\}$. With the notation defined, we have the material necessary to define a precompact expansion.   
\begin{definition}
	We let $K^* = \{(\AA,\vec{R}) :\AA \in \KK \}$. We define an expansion of $\AA\in \KK$ to be an element $\AA^* \in \KK^*$ which is isomorphic to $\AA$ when we remove its interpretations of the relation $\vec{R}$ (${\AA^*}\upharpoonright{L} = \AA$ ). We say $\KK^*$ is a precompact expansion of $\KK$ when $\forall \AA \in \KK$, $\AA$ only has finitely many expansions in $\KK^*$.
\end{definition}
The reason for the term precompact is fitting. Not only can it be viewed as any $\AA$ being covered by only finitely many expansions, but it can also be shown to relate to actual precompactness. Namely, if $\KK$ and $\KK^*$ are Fra\"iss\'e with limits $\FF$ and $\FF^*$ respectively, then the quotient $\Aut(\FF)/\Aut(\FF^*)$ is precompact.  
\begin{definition}
	Suppose $\FF$ is Fra\"iss\'e and $\FF^*$ is a precompact expansion of $\FF$. We say $\Age(\FF^*)$ has the expansion property (EP) relative to $\Age(\FF)$ if $\forall \AA \in \Age(\FF)$ $\exists \BB \in \Age(\FF)$ such that any expansion of $\AA$ embeds in to any expansion of $\BB$.     
\end{definition}
\begin{theorem}
	Let $\KK^*$ be an expansion of $\KK$ satisfying HP,JEP and EP relative to $\KK$. Then if the Ramsey degree of any $\AA \in \KK$ is bounded by the number of expansion in $\KK^*$, then $\KK^*$ has the Ramsey property.
\end{theorem}
We will generalize this later in a categorical context. \\
\\
We will conclude this section with a list of Fra\"iss\'e classes we are interested in and known facts about their Ramsey degrees. 
\begin{itemize}
	\item The class $\Gra$ is Fra\"iss\'e with the Random graph as it's universal homogeneous structure. It has finite Ramsey degrees
	\item The class $\LOGra$ is Fra\"iss\'e and has RP. 
	\item The class of all finite linear orders $\mathbf{LO}$ is Fra\"iss\'e with universal homogeneous structure $\QQ$. It has RP and finite big Ramsey degrees with $T(\AA) = \tan^{ (2|\AA| -1)}(0)$. 
	\item The class of partitioned linear orders with n-components $\mathbf{LO}_n$ is Fra\"iss\'e with universal homogeneous structure $\QQ_n$. It has RP and finite big Ramsey degrees with $T(\AA) = \tan^{ (2|\AA| -1)}(0)$.
	\item The class of all finite tournaments $\mathbf{Tour}$ is Fra\"iss\'e with universal homogeneous structure $\SSSS(2)$. It has finite (both big and small) Ramsey degrees with $t(\AA) = 
	\frac{2 |\AA|}{|Aut(\AA)|} $ $T(\AA) =\frac{2 |\AA|}{|Aut(\AA)|} \tan^{ (2|\AA| -1)}(0)   $ 
\end{itemize}
\section{Functors in Relation to Topological Dynamics}
\subsection{Forgetful Functors and Expansions}
Forgetful functors are quite abstract, and so finding a definition that is universally accepted can be difficult. But generally speaking, a forgetful functor should describe away to take a structure, and map it to one that somehow has less structure, and consequently less morphisms. We stick with the definition outlined in \cite{BigRamsey}.  
\begin{definition}
	Suppose $U:\KK_1 \rightarrow \KK_2$ is a functor. We say it is forgetful if
	\begin{itemize}
		\item $U: \text{hom}_{\KK_1} (\AA, \BB) \rightarrow \text{hom}_{\KK_2} (U(\AA),U(\BB) )$ is injective. (C)  
	\end{itemize} 
\end{definition}
This leads to our notion of an expansion. Our definition serves to reduce the concept of a relational expansions to its core algebraic/combinatoric properties.
\begin{definition}
	Suppose $U:\KK_1 \rightarrow \KK_2$ is a forgetful functor. We say it is an expansion if
	\begin{itemize}
		\item $U$ is surjective on objects. (Proj) 
		\item If $f\in \text{hom}_{\KK_2}(\AA, \BB ) $ and $U(\BB^*) = \BB$, then there is a unique $\AA^*$ with $U(\AA^*) = \AA$ and $f^*\in \text{hom}_{\KK_1}(\AA^*,\BB^*) $  such that $U(f^*)\circ g = f$ for some $g\in \text{hom}(\AA,\AA) = \text{Aut}(\AA) $ (Ref) 
	\end{itemize} 
\end{definition}
The (Proj) condition ensures that $U$ is projective on objects i.e every object in $\KK_2$ has an expansion in $\KK_1$. The condition (C) mimics how objects become less rigid when relations are removed. The final condition (and arguably one of the most important) is the reflective condition (Ref). This guarantees that if I have an object $\BB$ in $\KK_2$ and an embedding from $\AA$ into $\BB$, then by adding structure on to $\BB^*$, we induce a new structure on to $\AA$. In particular, if $\AA^\prime \in \binom{\BB}{\AA}$ and $f \in \text{hom}(\AA,\BB)$ has image $\AA^\prime $ with $U(f^*) = f$ for some $f^* \in \text{hom}(\AA^*,\BB^*)$, then we denote the image $(\AA^*)^\prime \in \binom{\BB^*}{\AA^*} $ and say that ``$\AA^\prime $ is supported by $(\AA^*)^\prime $".     
\\
\\
If you consider a class $\KK$ of finite structures with signature $\mathcal{L}$ and consider a relational expansion $\KK^*$ with signature $\mathcal{L}^* = \mathcal{L}\cup \{\vec{R}^* \}$, then the map $U: \KK^* \rightarrow \KK$  $U((\AA,\vec{R}^*)) = \AA$ is an expansion. In fact, with this in mind, we can generalize the notion of a precompact expansion.   
\begin{definition}
	Suppose $\KK_1$ and $\KK_2$ are categories of finite structures. We say $\KK_1$ is a categorical precompact expansion of $\KK_2$ if it admits an expansion $U:\KK_1 \rightarrow \KK_2$ with the property that $\forall \AA \in \KK_2$ there is finitely many $\AA^i \in \KK_1$ with $U(\AA^i) = \AA$.
\end{definition}
When such a functor is defined, we will define $m(\AA)$ to be the number of expansions of $\AA$ in $\KK_1$. It should be clear that a precompact expansion is also a categorical precompact expansion. When speaking of a categorical precompact expansion, we will always refer to expansions of $\AA$ as $\AA^i$ where $i\in \{1,...,m(\AA)\}$ or as $\AA^*$. The former case will be when it is necessary to speak of multiple expansions, the latter in the instance when any expansion will serve our purpose. This is not to be confused with $\AA_i$ which will always refer to a sequence of objects in a category.
\subsection{Proof of Theorem 1.1}
 In order to connect categorical precompact expansions to topological dynamics, we need to find a natural way to connect our functor to infinite structures, not just finite ones. Doing this requires we utilize the works in \cite{Kubis}. 
 \begin{definition}
 	Given a class of structures $\KK$, we defined $\sigma \KK = \{ (\AA_\alpha)_{\alpha \in \NN} :  \AA_\alpha \in \KK ,\forall \alpha <\beta  \exists a_{\alpha}^\beta \in \text{hom}(\AA_\alpha,\AA_\beta), \forall \alpha <\beta <\gamma, a_{\beta}^\gamma \circ a_{\alpha}^\beta = a_\alpha^\gamma   \}$. That is, $\sigma \KK$ is the collection of all infinite sequences from $\KK$ that admit a gluing matrix $f_\alpha^\beta$.
 \end{definition} 
 \begin{prop}
 	$\sigma \KK$ is a category when endowed with homomorphisms $\text{hom} ( (\AA_\alpha), (\BB_\alpha) ) = \{ (f_\alpha)_{\alpha \in \NN} : \exists \beta_\alpha \text{strictly increasing sequence such that } f_\alpha \in \text{hom}( \AA_\alpha, \BB_{\beta_\alpha}) \} $
 \end{prop}
 Having $\sigma \KK$ will be quite useful for us. For one, $\KK$ can be viewed as a subcategory of $\sigma\KK$  by identifying $\AA \in \KK$ with the constant sequence $(\AA_\alpha)$ with $\AA_\alpha = \AA $ for all $\alpha \in \NN$. The existence of a gluing matrix allows us to construct a direct limit $\lim_{\rightarrow} \AA_\alpha $ for every sequence in $\sigma \KK$. For any structure $\vec{\AA} \in \sigma \KK$, we always define the gluing matrix with the lowercase letter (in this instance $a_\alpha^\beta$). Moreover, it is clear that $\Age(\lim_{\rightarrow} \AA_\alpha )\subseteq \KK$, and any structure who's age is a subset of $\KK$ can be constructed via a sequence in $\sigma \KK$. The problem with $\sigma \KK$ is that multiple sequences can have the same limit. We solve this by defining a congruence on $\sigma \KK$. Take $\vec{\AA}, \vec{\BB} \in \sigma \KK$ and $\vec{f},\vec{g} \in \text{hom} (\vec{\AA},\vec{\BB})$. Let $\phi, \psi: \NN \rightarrow \NN$ be the subsequences with respect to $\vec{f}$ and $\vec{g}$ respectively. We say $f\sim g$ if both
 \begin{align*}
 \forall \alpha \exists \beta \geq \alpha \text{ such that } \phi(\alpha) \leq \psi(\beta) \; b_{\phi(\alpha)}^{\psi(\beta)} \circ f_\alpha = g_\beta \circ a_{\alpha}^\beta \\
 \forall \alpha \exists \beta \geq \alpha \text{ such that } \psi(\alpha) \leq \phi(\beta) \; b_{\psi(\alpha)}^{\phi(\beta)} \circ g_\alpha = f_\beta \circ a_\alpha^\beta
 \end{align*}
 It is clear that $\sim$ is an equivalence relation. Moreover, $\sim$ naturally allows us to define a quotient category $\sigma \KK / \sim$, where every arrow is an equivalence class. It also equates structures as well. Namely, if $(A_{\alpha}) \in \sigma \KK$ and $(\AA_{\psi(\alpha)})$ a subsequence, then $(\AA_{\psi(\alpha)})$ and $(\AA_\alpha) $ are treated as the same object in $\sigma \KK / \sim$. From this point foreward, when we talk of $\sigma \KK$, we are going to mean $\sigma \KK/\sim $. Another thing to note about $\sigma \KK$ is that now $\KK$ can be viewed as a subcategory with no new arrows in $\sigma \KK$. Moreover, colimits are still preserved (as they are preserved under subsequences). So, the statement $\lim\limits_{\rightarrow} \vec{\AA} $ is well defined despite $\vec{\AA}$ being an equivalence class.\\
 \\
 \begin{definition}
 	We say $\vec{\FF} \in \sigma \KK$ is a Fra\"iss\'e sequence if $\lim\limits_{\rightarrow} \vec{\FF}$ is Fra\"iss\'e.
 \end{definition}
 \begin{definition}
 	Suppose $F: \KK \rightarrow \mathcal{J}$ is a functor. Then, $F$ extends to a functor $F:\sigma \KK \rightarrow \sigma \mathcal{J} $ by defining $F( (\AA_\alpha)) = (F(\AA_\alpha)) $. 
 \end{definition}
 \begin{lemma}
 	If $U: \KK^* \rightarrow \KK$ is a categorical precompact expansion, then $U:\sigma \KK^* \rightarrow \sigma \KK$ satisfies (Proj) and (C). 
 \end{lemma}
 \begin{proof}
 	First, we show that $U$ saitsfies (Proj). Take $(\AA_i) \in \sigma \KK $ with gluing matrix $a_n^m $. We may assume that $(\AA_i)$ is never constant as we only need to check the elements in $\sigma \KK \setminus \KK $. Construct a tree like so. Let $T_i =   \{\AA^* \in \KK^* : U(\AA^*) = \AA_i  \} $ and $T_0 = \{\emptyset \}$. Let $T=\bigcup\limits_{i=0}^\infty T_i$ For all $\AA^* \in T_1$, we say $\emptyset \leq_T \AA^*$. Suppose we have $\leq_T$ defined on $\bigcup\limits_{i=0}^n T_i$ and $(\bigcup\limits_{i=0}^n T_i, \leq_T)$ is a tree and if $i\geq 1$, $\BB^*\in T_{i+1}$ extends $\AA^* \in T_i$, then $\exists g\in \text{hom} (\AA^*,\BB^*)$ such that $U(g)\circ \iota = a_{i}^{i+1} $ for some $\iota \in \text{Aut}(\AA_i) $. For each $\BB^* \in T_{n+1}$, find a $\AA^* \in T_n$, $g \in \text{hom}(\AA^*, \BB^*)$ and $\iota \in \text{Aut}(\AA_n) $ such that $U(g) = a_{n}^{n+1}$. Set $\AA^* \leq_T \BB^*$. Since we have only defined one unique predecessor for each $\BB^*$, it is clear that $(\bigcup\limits_{i=0}^{n+1} T_i, \leq_T) $ is a tree with the desired condition. Following this construction, we get a tree $\mathbf{T}=(T,\leq_T)$ that is countable and is locally finite. In fact, each $T_n$ is finite. \\
 	\\
 	By K\"onig's Lemma, there is a path $\AA_i^* $ in $\mathbf{T}$. This path comes along with a sequence $g_n^{n+1} \in \text{hom}(\AA_n^*,\AA_{n+1}^*) $ and $\iota_n \in \text{Aut}(\AA_n) $ such that $U(g_n^{n+1})\circ \iota_n = a_n^{n+1}$ and the condition that $U(\AA_n^*) = \AA_n $. It is clear that $(\AA_n^*) \in \sigma \KK^*$ with matrix induced by $g_n^{n+1}$ and $ U ((\AA_n^*)) = (\AA_n)$ as the sequences $(\iota_n)$ is an isomorphism of $U((\AA_n^*)) $ with $(\AA_n) $. Thus, $U$ satisfies (Proj).\\
 	\\
 	It is clear that $ U$ satisfies (C). Suppose we have a maps $\vec{f}, \vec{g} \in \text{hom}_{\KK^*} ( (\AA_n),(\BB_n)) $ with $U(\vec{f}) \sim U(\vec{g}) $. So, there is subsequences $ \psi $ and $\phi$ for which $\forall \alpha \exists \beta \geq \alpha $ such that $\phi(\alpha) \leq \psi(\beta) $ or $ \phi(\alpha) \geq \psi(\beta)$ and one of the diagrams commutes
 	\\
 	\\
 	\begin{tikzcd}
 	U(\AA_{\alpha}) \arrow[d, "U(a_{\alpha}^\beta)"] \arrow[r, "U(f_\alpha)"] & U(\BB_{\phi(\alpha)}) \arrow[d, "U(b_{\phi(\alpha)}^{\psi(\beta)})"]  \\
 	U(\AA_\beta) \arrow[r, "U(g_\beta)"] & U(\BB_{\psi(\beta)}) 
 	\end{tikzcd}
 	\\
 	\\
 	\begin{tikzcd}
 	U(\AA_{\alpha}) \arrow[d, "U(a_{\alpha}^\beta)"] \arrow[r, "U(g_\alpha)"] & U(\BB_{\psi(\alpha)}) \arrow[d, "U(b_{\psi(\alpha)}^{\phi(\beta)})"]  \\
 	U(\AA_\beta) \arrow[r, "U(g_\beta)"] & U(\BB_{\phi(\beta)}) 
 	\end{tikzcd}
 	\\
 	\\
 	In either case, as $U$ satisfies (C) on $\KK^*$, it follows that \\
 	\\
 	\begin{tikzcd}
 	\AA_{\alpha}\arrow[d, "a_{\alpha}^\beta"] \arrow[r, "f_\alpha"] & \BB_{\phi(\alpha)}\arrow[d, "b_{\phi(\alpha)}^{\psi(\beta)}"]  \\
 	\AA_\beta \arrow[r, "g_\beta"] & \BB_{\psi(\beta)}
 	\end{tikzcd}
 	\\
 	\\
 	or
 	\\
 	\\
 	\begin{tikzcd}
 	\AA_{\alpha} \arrow[d, "a_{\alpha}^\beta"] \arrow[r, "g_\alpha"] & \BB_{\psi(\alpha)} \arrow[d, "b_{\psi(\alpha)}^{\phi(\beta)}"]  \\
 	\AA_\beta \arrow[r, "g_\beta"] & \BB_{\phi(\beta)} 
 	\end{tikzcd}
 	\\
 	\\
 	also commutes, so $\vec{f}\sim \vec{g} $.
 \end{proof}
 The property (Ref) may appear lost, but it can also be recovered.
 \begin{lemma}
 	If $U:\KK^* \rightarrow \KK $ is a categorical precompact expansion, then $U$ satisfies (Ref) for arrows in $\text{hom}(\AA,\BB)$ such that $\AA \in \KK $.
 \end{lemma}
 \begin{proof}
 	If $\BB \in \KK$ the answer is trivial. Else, $\BB$ can be identified as an equivalence class of sequences $(\BB_n)$ with $\BB_n \in \KK $ and gluing matrix $\alpha_n^m $. Moreover, $f\in \text{hom}(\AA,\BB)$ is a sequence of the form $f_n \in \text{hom}(\AA, \BB_n)$ such that $f_m= \alpha_n^m\circ f_n$. By (Proj) there is $\BB_n^* \in \KK^*$ and $\beta_n^m \in \text{hom}(\BB_n^*, \BB_m^*)$ such that $U(\BB_n^*) = \BB_n$ and $U(\beta_n^m) = \alpha_n^m $. So, by (Ref) on $\KK$, there is an expansion $\AA^*$ of $\AA$ and $f_1^*\in \text{hom}(\AA^*, \BB_1^*) $ such that $U(f_1*)= f_1$. By defining $f_n^* = \beta_1^n \circ f_1^*$, we get $U(f_n^*) = \alpha_1^n \circ f_1 = f_n$ as desired. 
 \end{proof}
 Using Lemma 2.1 and 2.2, we get the following. 
 \begin{theorem}
 	If $U:\KK^* \rightarrow \KK $ is a categorical precompact expansion, $U: \sigma \KK^* \rightarrow \sigma \KK$ is an expansion functor.
 \end{theorem}
 \begin{proof}
 	By Lemma 2.1 and 2.2, it suffices to show that (Ref) holds for $\AA,\BB \in \sigma \KK \setminus \KK$.\\
 	\\
 	Take $f\in \text{hom}(\AA,\BB)$. Let $\alpha_n^m $ $\gamma_n^m $ be the gluing matrix of $(\AA_n) $ and $ (\BB_n)$ respectively. Note, $f$ can be identified as a sequence $f_n\in \text{hom}(\AA_n, \BB)$. Let $\BB^*$ be an expansion of $\BB$. Doing the same trick as in lemma 2.1, we can construct a tree $T$ with $T_0 = {\emptyset}$, $T_i = \{\AA^*\in \KK^* : U(\AA^*) = \AA_i \; \exists f_i^* \in \text{hom}(\AA_i^*,\BB_n^*) \; U(f_i^*)=f_i\} $ and $T= \bigcup\limits_{i=0}^\infty T_i$. Note, by the above lemma, no $T_i$ is empty. For each $\AA_{i+1}^* \in T_{i+1}$, $\exists \AA_i^* \in T_i $ such that $\exists \beta_{i}^{i+1} \in \text{hom}(\AA_i^*,\AA_{i+1}^*)$ $U(\beta_i^{i+1}) =\alpha_i^{i+1}$. Choose exactly one $\AA_i^*$ for each $\AA_{i+1}^*$ and define $\AA_i^* \leq_T \AA_{i+1}$. Following this recursive construction, we create a locally finite tree $T$. Find a branch $(\AA_n^*)$ and let $f_n^* \in \text{hom}(\AA_n^*, \BB^*)$ be such that $U(f_n^*) = f_n$. Then, $(f_n^*) \in \text{hom}(\AA^*,\BB^*)$. Suppose $\eta_n^m \in \text{hom}(\BB_n^*,\BB_m^*)$ such that $U(\eta_n^m) = \gamma_n^m$. Since
 	\begin{align*}
 	U(\eta_n^m \circ f_n^*) & = \gamma_n^m \circ f_n\\
 	&= f_m \circ \alpha_n^m\\
 	&= U(f_m^*) \circ U(\beta_n^m)\\
 	&=U(f_m^* \circ \beta_n^m)
 	\end{align*}
 	and $U$ is injective, $\eta_n^m \circ f_n^*=f_m^* \circ \beta_n^m $ and so $f^* =(f_n^*) \in \text{hom}(\AA^*, \BB^*)$.
 \end{proof}
  It is not enough that we maintained all the properties of an expansion. While this is an incredibly useful property, we need to relate automorphism groups of Fra\"iss\'e structures. This means that we need to ensure that $\text{hom}(\vec{\FF},\vec{\FF})$ corresponds to structural embeddings as we would want them to, else our construction may not provide useful information. On top of this, we would hope that our functor sends Fra\"iss\'e structures to other Fra\"iss\'e structures. Else, there would be no natural way to compare the automorphism group of one to another. Our next lemmas show that we are indeed in a best case scenario for comparing automorphism groups.  
 \begin{lemma}
 	Suppose $\FF$ is the Fra\"iss\'e limit of a class $\KK$ and $\vec{\FF} \in \sigma \KK$ is a sequence for $\FF$. Then, $hom(\FF,\FF) \cong hom(\vec{\FF}, \vec{\FF})$.
 \end{lemma}
 \begin{proof}
 	Note, this is not immediately trivial as $\text{hom}(\FF,\FF)$ is a collection of model embeddings from $\FF$ in to itself, while $hom(\vec{\FF}, \vec{\FF}) $ is a collection of equivalence classes of arrows. Take $f \in \text{hom}(\FF,\FF) $. We will construct a sequence of functions $f_n: \AA_n \rightarrow \BB_n $ with $\AA_n, \BB_n \in \KK$ with gluing matrices $a_{n}^m\in \text{hom}(\AA_n, \AA_m) $ and $b_{n}^m\in \text{hom}(\BB_n, \BB_m) $ such that $(\AA_n)$ and $(\BB_n)$ limit to $\FF$. Moreover, $ \vec{f}$ limits to $f$ in a natural way.\\
 	\\
 	We start by assuming $\FF$ has universe $\NN$. We define $\FF\upharpoonright{[n]} = \AA_n$. It is clear that $f_n = f\upharpoonright{[n]}$ is a map $f_n: \AA_n \rightarrow \FF$. Since $n$ is finite, $M_n = \text{max}_{i\in [n]} \{ f(i) \} $ exists and $M_n$ is an increasing sequence. Setting $\BB_n = \FF_{M_n}$, we see that $f_n: \AA_n \rightarrow \BB_n$ is a model embedding. Moreover, $a_{n}^m $ and $b_n^m$ taken to be the inclusion maps grants that $\lim\limits_{\rightarrow}  \AA_n = \bigcup\limits_{n=1}^\infty \AA_n = \FF$ and similarly for $\BB_n$. Hence, $\vec{f} \in \text{hom} ( (\AA_n),(\BB_n)) $, but 
 	$\text{hom} ( (\AA_n),(\BB_n))= \text{hom}(\FF,\FF)$ in $\sigma \KK$ on the account of $\AA_n$ and $\BB_n$ having the same colimit (moreover, on the account that $\BB_n $ is a subsequence of $\AA_n$). This mapping from $f$ to $\vec{f}$ is unique up to $\sim$. Suppose $f \neq g$ where $g\in \text{hom} (\FF,\FF) $. If $\vec{f}\sim \vec{g} $, then there are strictly increasing sequences $\phi, \psi : \NN \rightarrow \NN$ such that $f\upharpoonright{\phi(n)} = g\upharpoonright{\psi(n)}$ for cofinally many $n$. But then clearly,  $f=g$. So, this mapping $f\rightarrow \vec{f}$ is an injection that shows $\text{hom}(\FF,\FF) \subseteq \text{hom}(\vec{\FF},\vec{\FF}) $. We also claim it is surjective.\\
 	\\
 	Take an equivalence class $\vec{f} \in \text{hom} (\vec{\FF}, \vec{\FF})$. We may assume that $\vec{\FF} = (\FF\upharpoonright{[n]})$ as (proven in \cite{Kubis}) all Fra\"iss\'e sequences are isomorphic. Then of course, any $\vec{f} \in \text{hom}(\vec{\FF}, \vec{\FF})$ is a sequence of partial functions on $\NN$ ordered by inclusion, and $\bigcup\limits_{n=1}^\infty f_n = f : \FF \rightarrow \FF$ and gets mapped to $\vec{f}$ up to equivalence. 
 \end{proof}
 \begin{lemma}
 	Suppose $\KK^*$ is a categorical precompact expansion of $\KK$. If $f\in \text{hom}(\mathcal{G}, \FF)$ where $\mathcal{G},\FF \in \sigma \KK$, then $U(f\upharpoonright{\AA} ) = U(f) \upharpoonright{U(\AA)} $
 \end{lemma}
 \begin{proof}
 	This is a quite trivial consequence of our construction. Take $(f_n)$ a representative of of $f$ with $\text{dom}(f_1) = \AA$ so that $f\upharpoonright{\AA} = f_1$. Then, $U(f)$ has $(U(f_n))$ as a representative and of course, $U(f) \upharpoonright{U(\AA)} = U(f_1) $ as desired.
 \end{proof}
This leads us to our final lemma. After showing this result, we will have all the material necessary to connect functors to automorphism groups and consequently prove Theorem [].
 \begin{lemma}
 	If $U:\KK^* \rightarrow \KK$ is a categorical precompact expansion of Fra\"iss\'e classes with Fra\"iss\'e objects $\FF^* $ and $\FF$, then $U(\FF^*) = \FF$.
 \end{lemma}
 \begin{proof}
 	It suffices to show any $f\in \text{hom}(\AA, U(\FF^*)) $ with $\AA \in \KK$ can extend to an automorphism. Take a reflection $f^* \in \text{hom}(\AA^*, \FF^*) $. As $\FF^*$ is Fra\"iss\'e, $f^*$ extends to an automorphism $\hat{f}^*$. But then as functors send invertible elements to invertible elements, $ U(\hat{f}^*)$ is an automorphism of $U(\FF^*)$ and by lemma 2.5 $U(\hat{f}^*)\upharpoonright{\AA} = U(\hat{f}^* \upharpoonright{\AA^*}) = f$ as desired.
 \end{proof}
 \begin{theorem}
 	Suppose $\KK^*$ is a categorical precompact expansion of $\KK$ witnessed by functor $U$. Suppose $\KK$ and $\KK^*$ are Fra\"iss\'e with limits $\FF$ and $\FF^*$ respectively. Then, there is a injective continuous group homomorphism from $G^*=\Aut(\FF^*) $ to $G=\Aut(\FF)$. 
 \end{theorem}
 \begin{proof}
 	To show this, we need a candidate for a map between the two automorphism groups. The natural thing to do is use $U$. By lemma 2.1, $U:\text{hom}(\FF^*, \FF^*) \rightarrow \text{hom}(\FF,\FF) $ is injective. Note that $\Aut(\FF^*) \subseteq \text{hom} (\FF^*,\FF^*)$. I claim that if $f$ is invertible, then $U(f)$ is invertible. As $f$ is invertible, there is an $f^{-1}$ for which $f\circ f^{-1} = \text{id}_{\FF^*}$. So, $U(f\circ f^{-1}) = U(f)\circ U(f^{-1}) = \text{id}_{\FF}$ and so $ U(f^{-1}) = U(f)^{-1}$ as one might expect. Therefore, $U:\Aut(\FF^*) \rightarrow \Aut(\FF)$ is a group embedding. We now show it is continuous. We do this by utilizing Lemma 2.2. Any $f\in \Aut(\FF^*)$ and be identified with a sequence of functions $f_n : \FF_n^* \rightarrow \FF^* $. So, a sequence of automorphisms $f^i \in \Aut(\FF^*) $ converging to $f$ would mean that for any $N$, there is an $M$ such that for all $i \geq M$ and $\forall n\leq N $ the approximations $f_n^i: \FF_n^* \rightarrow \FF^*$ remain constant. But of course, this means that $U(f_n^i)$ remains constant and agrees with $U(f_n)$, so $U(f^i) \rightarrow U(f)$ pointwise.\\
 	\\
 	Now we show the image of $G^*$ under this mapping is closed. Let $f_n \in \{ U(g) : g\in G^* \}$ be convergent with limit $f$. Suppose $\FF $ has universe $\NN$ and let $\FF_n$ be the model induced by restricting $\FF$ to $\{1,...,n\} \subseteq \NN$. As $f_n$ is convergent, we can assume that $f_n \upharpoonright{\FF_{n-1}} = f_{n-1} \upharpoonright{\FF_{n-1}} $ by passing to a subsequence. As our expansion is precompact, by the pigeon hole principle, there is a subsequence $f_{n,1}$ of $f_n$ such that $f_{n,1}\upharpoonright{\FF_1} $ admits a sequence of reflections $g_{n,1}^* \in \text{hom} (\FF_1^*, \FF^*)$  with $U(g_{n,1}^*) = f_{n,1} \upharpoonright{\FF_1}$ for a fixed expansion $\FF_1^* $. Note, we need not worry about a corrective automorphism of $\FF_1$ as we assumed that $f_n$ was in the image of $U$. Recursively construct $f_{n,k+1}$ and $g_{n,k+1}$ such that $f_{n,k+1} $ is a subsequence of $f_{n,k}$, $g_{n,k+1} \in \text{hom} (\FF_{k+1}^* , \FF^*)$ with $U(g_{n,k+1}) = f_{n,k+1} \upharpoonright{\FF_{k+1}} $. As $\FF^*$ is Fra\"iss\'e, we can construct a sequence $g_n^*\in G^*$ by extending the map $g_{n,n}$ i.e $ g_n^* \upharpoonright{\FF_n^*} = g_{n,n}$. But of course, by the way we constructed $g_{n,n}$, $g_n^*\upharpoonright{ \FF_{n-1}^*} = g_{n-1}^* \upharpoonright{\FF_{n-1}^*} $. So, $g_n^*$ is a Cauchy sequence in $G^* $ which is complete. Consequenty, there is a limit $g^* \in G$. But then by our previous lemma we have
 	\begin{align*}
 	U(g^*) \upharpoonright{\FF_n}  &= U(g^* \upharpoonright{\FF_n^*} )\\
 	&= U(g_{n,n})\\
 	&= f_{n,n}\upharpoonright{\FF_n} 
 	\end{align*}
 	In particular, $f_n$ admits a sequence that converges to $U(g^*)$ so $f_n$ must converge to $U(g^*)$. Consequently, $f_n$ converges to something in the image of $G^* $ under the map $U$ and hence the image is closed. 
 \end{proof}
 \textbf{Proof of Theorem 1.1:} Let $u(G^*) = \{ U(g) : g\in G^* \}$. From the above we have that $ u(G^*)$ is closed, but what can we say about the quotient space $G/u(G^*)$? Basic open sets of $G$ are of the form
 \begin{align*}
 V_{g,\AA_n} &= \{ f \in G :  f\upharpoonright{\AA_n} = g\upharpoonright{\AA_n} \} 
 \end{align*}
 where $\AA_n$ is the finite model gotten by restricting $\FF$ to $[n]$. For a given n, let us define sets for $i\leq m(\AA_n) $
 \begin{align*}
 S_i &= \{ f\in G : \exists h \in \text{hom}(\AA_n^{i} , \FF^*) \exists \iota \in \text{Aut}(\AA_n) \; U(h)\circ \iota  = f \upharpoonright{\AA_n} \}
 \end{align*}
 These sets are nonempty by the properties of $U$.  Moreover, realizing precompactness is a consequence of the following observations.\\
 \\
 \textbf{Observation 1:} if $f \in S_i $ with $U(h) \circ \iota =f \upharpoonright \AA_n $, then by extending $\iota $ to a full automorphism $\hat{\iota} \in G $, we get $(U(h)\circ \hat{\iota}) \circ f^{-1} \upharpoonright \AA_n = \text{id}_{\AA_n} $. In particular, $(U(h)\circ \hat{\iota}) \circ f^{-1}\in V_{\text{id}_{\AA_n} , \AA_n} $.\\
 \\
 \textbf{Observation 2:} Taking $h_i\in \text{hom}(\AA_n^i,\FF^*)$ arbitrary and extending $U(h_i)$ to a full automorphism $f_i\in G $, then $f_i\in S_i $ and $f_i\upharpoonright\AA_n = U(h) $.
 \\
 \\
 \textbf{Observation 3:} Given any  $g \in S_i$ with 
 \begin{align*}
 U(h_1)\circ \iota &= g\upharpoonright \AA_n
 \end{align*}
 then there is a $T^*\in G^*$ such that $T^*\circ h = h_1 $. Consequently,
 \begin{align*}
 U(T^* \circ h) \circ \iota &= g\upharpoonright \AA_n\\
 U(T^*) \circ f \circ \iota &= g\upharpoonright \AA_n
 \end{align*} 
 Let $F$ be a set of selections according to Observation 2, that is,  for each $i$, $F$ contains exactly one $f_i \in S_i $ that satisfies the condition of Observation 2. Then $u(G^*) F V_{\text{id}, \AA_n} = G$. To see this, take any $g\in G$. It must be the case that $g\in S_i$ for some fixed $i$. By Observation 3, there is a $T^*\in G^*$ such that $U(T^*) \circ f_i \circ \iota = g\upharpoonright \AA_n $. By Observation 1, there is a $\hat{\iota}\in V_{\text{id}_{\AA_n} ,\AA_n}$ such that $U(T^*) \circ f_i \circ \hat{\iota} = g$. 
 \\
 \\
 Hence $G/u(G^*)$ is a precompact metrizable $G$-space. Consequently, $M(G)$ is metrizable.
\section{Bounding Ramsey Degrees}
We have now shown that if $\KK^*$ is a categorical precompact expansion of $\KK$ with finite Ramsey degrees, then $\KK$ also has finite Ramsey degrees. Now our goal is to develop techniques to best compute Ramsey degrees given this knowledge. Moreover, under very strict conditions, we can also compute big Ramsey degrees. This will be of use to us in the final section, where we apply all of our developed machinery. 
\subsection{Ramsey Degrees}
We start with bounding Ramsey degrees from above. This will likely be the most practical use of categorical precompact expansions if ever applied to other categories. 
\begin{theorem}
	Suppose $\KK_1$ is a categorical precompact expansion of $\KK_2$. If for $\AA\in \KK_2 $, $\text{max}_{i<m(\AA)} t_{\KK_1} (\AA^i) < \infty, $ then $t_{\KK_2} (\AA) \leq  \sum\limits_{i< m(\AA)} t_{\KK_1}( \AA^i)$.
\end{theorem}
\begin{proof}
	Suppose we have $\AA \in \KK_2$ is as above and take $\BB \in \KK_2$. It suffices to work with the case $m(\AA) =2$ as the argument used can be recursively applied (like in [2]). Let $t=\sum\limits_{i=0}^1 t_{\KK_2}( \AA^i)  $. We start by constructing a $\CC$ for which $\CC\rightarrow (\BB)_{k,t }^{\AA} $. Fix a chain of $\CC_i \in \KK_1 $, $i \leq 2 $ such that
	\begin{align*}
	U(\CC_0) &=  \BB \\
	\CC_{1} &\rightarrow (\CC_0)_{k, t_{\KK_1}(\AA^{0} )}^{ \AA^0}\\
	\CC_{2} &\rightarrow (\CC_1)_{k, t_{\KK_1}(\AA^{1} )}^{ \AA^1}
	\end{align*}
	I claim that $\CC = U(\CC_2)$ is as we desire. Define a structural colouring 
	\begin{align*}
	\chi : \text{hom}(\AA, \CC) &\rightarrow k
	\end{align*}
	We define two complimentary colourings. The first,
	\begin{align*}
	\chi_1: \text{hom}(\AA^2, \CC_2) &\rightarrow k \\
	\chi_1(f) &= \chi(U(f))
	\end{align*}
	Of course, there is an $f \in \text{hom}(\CC_1,\CC_2)$ such that $\chi_1$ takes at most $t_2$ many colours on $f\circ \text{hom}(\AA^2,\CC_1) $. Let $S_1\subseteq k$ be this set of colours.  Next, we define
	\begin{align*}
	\chi_2 :\text{hom}(\AA^1, \CC_1) &\rightarrow k\\
	\chi_2(h) &= \chi(U(f\circ h))
	\end{align*}
	So, there is a $g\in \text{hom}(\CC_0, \CC_1) $ such that $\chi_2$  takes $t_1$ many colours on $g\circ \text{hom}(\AA^1,\CC_0) $. Similarly, we let $S_2\subseteq k$ be this set of colours $\chi_2$ takes on this set. We claim that $U(f\circ g)$ is such that $\chi$ takes at most $t_1+t_2$ many colours on $U(f\circ g) \circ \text{hom}(\AA,\BB) $.\\
	\\
	There are two cases to consider. Take $h\in \text{hom}(\AA,\BB)$ and suppose there is an $h^* \in \text{hom}(\AA^1,\CC_0)$ such that $U(h^*)\circ \iota = h$ for some $\iota \in \text{Aut}(\AA) $. In this instance,
	\begin{align*}
	\chi(U(f\circ g ) \circ h ) &= \chi( U(f\circ g)\circ U( h^*) \circ \iota  ) \\
	&=\chi( U(f\circ g)\circ U( h^*)) \\
	&= \chi_2(g\circ h^*) \in S_2 \; \text{ as } g\circ h^* \in \text{hom}(\AA^1,\CC_1) 
	\end{align*}
	If there is no such $h^*$, by (Ref), there is an $h^*\in \text{hom}(\AA^2,\CC_0) $ such that $U(h^*)\circ \iota  = h$ for some $\iota \in \text{Aut}(\AA) $. We then have 
	\begin{align*}
	\chi(U(f\circ g) \circ h) &= \chi(U(f\circ g )\circ U( h^*  )\circ \iota) \\
	&=	\chi(U(f\circ g )\circ U( h^*  ) ) \\
	&= \chi_1(f\circ g\circ h^*  ) \in S_1
	\end{align*} 
	Consequently, the range of $\chi$ restricted to $U(f\circ g) \circ \text{hom}(\AA,\BB) $ only takes values in $S_1\cup S_2$ and $|S_1\cup S_2| \leq t$ as desired.    
\end{proof}
On its surface, this theorem may appear shallow. Consider the class of permuted graphs $\mathbf{PERGra} $, graphs with two orderings $(\GG,<,\prec) $. If we consider the canonical forgetful functor $U:\mathbf{PERGra} \rightarrow \LOGra $ via $U((\GG,<,\prec)) = (\GG,<) $, then we immediately see how crude the bound can be. Consider the cycle of length $n$ in $\LOGra$, call it $(\GG,<)$. There are exactly $\frac{(n-1)!}{2}$ many ways to add a linear order to it (there are exactly $n!$ many permutations of $n$ and the automorphisms of the n-cycle are isomorphic to the dihedral group which has order $2n$). At best, if $\mathbf{PERGra} $ has RP, this gives $t_{\LOGra}((\GG,<)) \leq \frac{(n-1)!}{2} $. But it is folklore that $\LOGra $ has RP. So this bound is not very helpful. Where this bound shines is where a precompact expansion might not be useful. For example, consider $\text{Age}(S(2))$. It was proved in [6], that any precompact expansion of $\text{Age}(S(2))$ could not lower the Ramsey degree of the triangle circuit below $2$. However, the above technique ensures that the Ramsey degree is exactly $2$ when paired with the following strengthening condition reminiscent of the expansion property. 
\begin{theorem}
	Let $\KK_1$ be a categorical precompact expansion of $\KK_2$ with forgetful functor $U:\KK_1\rightarrow \KK_2$. Suppose that $\forall \AA \in \KK_2$, $\exists \BB \in 
	\KK_2$ such that ${{\BB^i}\choose{\AA^j}} \neq \emptyset$ for all $i$ and $j$.
	If for $\AA\in \KK_2 $, $\text{max}_{i<m(\AA)} t_{\KK_1} (\AA^i) < \infty, $ then $m(\AA) \leq t_{\KK_2} (\AA) \leq  \sum\limits_{i< m(\AA)} t_{\KK_1}( \AA^i)$ (we say this $\BB$ satisfies the homogeneous condition).
\end{theorem}
\begin{proof}
	It is clear that the upper bound is satisfied by Theorem 1.2. So now, it suffices to show the lower bound. For a given $\AA\in \KK_2$, choose a $\BB \in \KK_2$ satisfying the homogeneous condition. We show that there is a $\chi : {{\CC}\choose{\AA}} \rightarrow \{1,..,t_{\KK_2} (\AA) \}$ that takes $t_{\KK_2} (\AA) $ many values on ${ {\tilde{\BB}\choose{\AA}} } $ for any $\tilde{\BB} \in {{\CC}\choose{\BB} } $. Take any extension of $\CC$ in $\KK_1$, say $\CC^1$. Given any $\tilde{\AA} \in  {{\CC}\choose{\AA}}$, it is supported by a unique $\AA^i $ in $\CC^1$. Consequently, the map $\chi(\tilde{\AA}) = i$ is as required.\\
	\\
	To see this, any $\tilde{B} \in {{\CC}\choose{\BB} }  $ is supported by some expansion $\BB^j$ in $\CC^1$. Moreover, each copy $\tilde{A} \in {{\tilde{\BB}}\choose{\AA} }$ is supported by a unique $\AA^i$. Moreover, by our hypothesis, for each $i<m(\AA)$, there is some $\tilde{A} \in{{\tilde{\BB}}\choose{\AA} } $ supported by $\AA^i$. This gives $m(\AA)$ as a lower bound. 
\end{proof}
In the case that $\KK_1$ has RP, it becomes immediate that $t_{\KK_2}(\AA) = m(\AA)$. 
\subsection{Big Ramsey Degrees}
Computing big Ramsey degrees is often a very nontrivial task. However, no study of a classes Ramsey properties is complete without some analysis of potential big Ramsey degrees. As we have seen, categorical precompact expansions also allow us to compare infinite objects and consequently, we can compute big Ramsey degrees under some very strict conditions. 
\begin{theorem}
	Suppose $\KK_1$ is a categorical precompact expansion of $\KK_2$. Suppose $\FF_1 $ and $\FF_2 $ are Fra\"iss\'e structures of $\KK_1$ and $\KK_2$ respectively. Moreover, suppose that 
	\begin{itemize}
		\item ${{\FF_1}\choose{\FF_1}} = {{\FF_2}\choose{\FF_2}}$ (equivalently,  $\FF_1$ is the only expansion of $\FF_2$)
	\end{itemize}
	If for $\AA\in \KK_2 $, $\text{max}_{i<m(\AA)} T_{\KK_1} (\AA^i) < \infty, $ then $T_{\KK_2} (\AA) =  \sum\limits_{i< m(\AA)} T_{\KK_1}( \AA^i)$. 
\end{theorem}
\begin{proof}
	Let $t=\sum\limits_{i< m(\AA)} T_{\KK_1}( \AA^i)$. First we show that $T_{\KK_2} (\AA) \leq t$, which only requires the first two assumptions. Take a colouring $\chi : {{\FF_2}\choose{\AA} } \rightarrow k$ with $k \geq t$. By lemma 1.1,
	\begin{align*}
	{{\FF_2}\choose{\AA} } &= \bigcup_{i<m(\AA)} {{\FF_1}\choose{\AA^i}}\\
	{{\FF_1}\choose{\FF_1}} &\subseteq {{\FF_2}\choose{\FF_2}} 
	\end{align*}
	So, $\chi$ induces a colouring on ${{\FF_1}\choose{\AA^i}} $ for each $i$. Start by finding an $\FF_1^{'} \in {{\FF_1}\choose{\FF_1}} $ such that $\chi\upharpoonright{{{\FF_1^{'}}\choose{\AA^0}}} $. Suppose by recursion, for some $k< m(\AA)$, we have $\FF_1^{(k)} \in {{\FF_1^{(k-1)}}\choose{\FF_1^{(k-1)} }} $ such that $\chi $ takes at most $T_{\KK_1}(\AA^j)$ many values on ${{\FF_1^{(k)}}\choose{\AA^j} } $ for $j\leq k$. If $k= m(\AA)-1$, we are done and as $\FF_1^{(k)} \in {{\FF_2}\choose{\FF_2}}$ and $\chi$ takes at most $t$ many values on ${{U(\FF_2^{(k)})}\choose{\AA}} $. If $k < m(\AA)-1$, then find $\FF_1^{(k+1)} \in {{\FF_1^{(k)}}\choose{\FF_1^{(k)}}}$ for which $\chi$ takes at most $T_{\KK_1}(\AA^{k})$ many values on ${{\FF_1^{(k)}}\choose{\AA^k}} $. This construction guarantees that we can construct $\FF_1^{(m(\AA) )} $ which witnesses $T_{\KK_2}(\AA) \leq t$ as desired.\\
	\\
	Next, we construct a colouring that witnesses $ T_{\KK_2}(\AA) \geq t$. This is not too complicated to do, given $\FF_1$ is the only expansion of $\FF_2$. That is because
	\begin{align*}
	{{\FF_1}\choose{\FF_1}} &= {{\FF_2}\choose{\FF_2}}
	\end{align*}
	For each ${{\FF_1}\choose{\AA^i}}$, we construct a function $\chi_i:{{\FF_1}\choose{\AA^i}} \rightarrow \{(0,i),...,(T_{\KK_1}(\AA^{i})-1,i) \} $ such that any $\FF_1^{'} \in { {\FF_1}\choose{\FF_1}}$ witnesses $\chi_i\upharpoonright{ {{\FF_1^{'}}\choose{\AA^i}  } } $ is surjective. Thus, we can define a function $\chi$ on ${{\FF_2}\choose{\AA} }$ by identifying ${{\FF_2}\choose{\AA} }$ with $\bigcup_{i<m(\AA)} {{\FF_1}\choose{\AA^i}}$ and setting $\chi\upharpoonright{{{\FF_1^{'}}\choose{\AA^i}}} = \chi_i$. The consequence is that for any $\FF_2^{'} \in { {\FF_2}\choose{\FF_2}} \subseteq  { {\FF_1}\choose{\FF_1}  }$, $\chi\upharpoonright{{{\FF_2^{'}}\choose{\AA}}}$ takes exactly $t$ many values. 
\end{proof} 
\section{Applications to $\SSSS(n) $}
In this section, we will finally put everything we've established so far to use. We will use the computed Ramsey degrees of objects from section 2.2, along with the machinery we've developed in section 3 to compute the Ramsey degrees of $\Age(\SSSS(n))$. But first, we must define what $\SSSS(n)$ is to begin with and deduce some properties it has. 
\subsection{The Structure $\SSSS(n) $}
\begin{definition}
	$\SSSS(n) = (S(n),\sigma_0,..,\sigma_{n-1})$ is the structure who's domain is a countably infinite dense subset of the unit circle with no two points making an angle $\frac{2\pi k}{n}$ ($k\in \ZZ$) and $\sigma_k$ is a binary relation with $\sigma_k(x,y)$ if and only if $\text{arg}( \frac{x}{y}) \in (\frac{2\pi k}{n}, \frac{2\pi (k+1)}{n})$.
\end{definition}
In the case of $n=2$, $\SSSS(n)$ can be viewed as a digraph, with $\sigma_0$ and $\sigma_1$ determining a direction. This is quite natural too because  $\neg \sigma_0(x,y) \iff \sigma_1(x,y)$. In fact, in this instance, $\SSSS(2)$ is the universal ultrahomogeneous tournament. Sadly, we do not have the exact same type of symmetry in the case of $n>2$, but instead we have the following.
\begin{lemma}
	Consider $\SSSS(n)$. We have $\sigma_k(x,y) \iff \sigma_{n-1-k}(y,x)$.
\end{lemma}
\begin{proof}
	We use the fact that $\text{arg}(\frac{x}{y}) = 2\pi - \text{arg}(\frac{y}{x})$. Consequently, \\
	$\frac{2\pi k}{n} < \text{arg}(\frac{x}{y}) < \frac{2\pi (k+1)}{n} $ if and only if $\frac{2\pi (n-1-k)}{n }  < \text{arg}(\frac{y}{x}) < \frac{2\pi (n-k)}{n}$.
\end{proof}
The above suggests that the case of $n$ odd is slightly different from the even case. In the even case, there is a one to one correspondents between relations $\sigma_k$. So, one can view $ \text{Age}(\SSSS(n))$ as a class of edge coloured tournaments. That is, any member $\AA \in \text{Age}(\SSSS(n)) $ is cryptomorphic to an edge coloured tournament. In the case $n$ is odd, $\sigma_{\frac{n-1}{2}} $ can be interpreted as the disjointness relation. This also guarantees that our $\SSSS(3) $ is the same $\SSSS(3)$ one can see in \cite{Zucker}.
\begin{definition}
	We define the class $\mathbf{Tour}_k$ to be the collection of structures $(\mathbf{X}, c)$ where $\mathbf{X} \in \mathbf{Tour}$ and $c$ is a $k$ colouring of the arrows of $\mathbf{X}$.
\end{definition}
\begin{theorem}
	Suppose $ n= 2k $ for some $k\in \NN$. Then, there exists an injective functor $F:  \text{Age}(\SSSS(n)) \rightarrow \mathbf{Tour}_k$.
\end{theorem}
\begin{proof}
	Consider a structure $\AA = (A, \sigma_0,..., \sigma_{n-1}) \in \text{Age}(\SSSS(n)) $. We send $\AA$ to a edge coloured digraph $\hat{\AA} = (A, \rightarrow_0,..,\rightarrow_{k-1}$ with $x\rightarrow_j y \iff \sigma_j(x,y)$. To see that $\hat{\AA}$ is an edge coloured tournament, take any distinct pair $x,y \in A$. There must be a $\sigma_j$ for which $\sigma_j(x,y)$. By our lemma, we may assume without loss of generality that $j<k$. So, $x\rightarrow_j y$. So, indeed $\hat{\AA} \in \text{Age}(\SSSS(n))$. Thus, $F(\AA) = \hat{\AA}$ is well defined. Moreover, by construction, $\text{hom}_{\text{Age}(\SSSS(n)) }(\AA,\BB) = \text{hom}_{\textbf{Tour}_k} ( F(\AA),F(\BB))$.\\
	\\
	I claim that this functor is not an isomorphism if $k>1$. To see this, consider the tournament $\tilde{\AA}= (\{a,b,c\}, \rightarrow_0, \rightarrow_1 )$ with $a\rightarrow_0 b $ $a\rightarrow_0 c$ and $b\rightarrow_1 c $. There is no such member of $\text{Age}(\SSSS(4)) $ that can represent it. To see this, split $S(4)$ along the standard $4$ quadrants in the Euclidean plane $Q_1,Q_2,Q_3$ and $Q_4$. We look for an $x,y,z \in S(4)$ that induces $\AA$. Without loss of generality, take $x\in Q_1$ with $\text{arg}(x) < \frac{\pi}{8}$. We need a $y$ and $z$ for which $\sigma_0(x,y)$ and $\sigma_0(x,z) $. However, this forces $\text{arg}(y), \text{arg}(z) > \frac{\pi}{8}$ and in either $Q_1$ or $Q_2$. No matter what we choose for $y$, it becomes impossible to choose a $z$ for which $\sigma_1(y,z) $ or $\sigma_1(z,y)$ as $\text{arg}(z) < \frac{\pi}{8} + \frac{\pi}{4}$. Similar constructions force $F$ to be strictly injective for $k>1$.
\end{proof}
As we can see from the above, we cannot use the exact some techniques used in \cite{ColouredMilliken}. In particular, we cannot identify $\Age(\SSSS(n))$ with $\mathbf{Tour}_k $, but rather as a subcategory due to metric restrictions on $\SSSS(n) $. But of course, this tells us nothing about the Ramsey properties of either class. To see this, one can view \cite{SokicPoset}, where a variety of classes of Posets are proven not to have RP despite being subclasses of classes with RP and vice versa. 
\subsection{Computation of Finite Ramsey Degrees}
Following the technique in \cite{ColouredMilliken}, we will show the class $\Age(\QQ_n) $ of partitioned linear orders is a categorical precompact expansion of $\Age(\SSSS(n)) $.  
\begin{definition}
	Let $\QQ_n = (\QQ, \mathcal{P}_1,...,\mathcal{P}_n)$ where $ \mathcal{P}_i$ form a partition of $\QQ$ with each $\mathcal{P}_i$ order isomorphic to $\QQ$. $\text{Age}(\QQ_n)$ is a Fraisse class with the Ramsey Property.
\end{definition}
Take an $\XX \in \text{Age}(\QQ_n)$. We show how one can construct a member of $\text{Age}(\SSSS(n))$ with it. Setting $\XX = (X,\mathcal{P}_1^{\XX},..,\mathcal{P}_n^{\XX})$, we define $p(\XX) = (X,\sigma_0^{p(\XX)},.., \sigma_{n-1}^{p(\XX)})$ with
\begin{align*}
(\forall i\leq j) (\forall x\in\mathcal{P}_i^{\XX}) (\forall y\in \mathcal{P}_j^{\XX})\;  x<^{\XX}y  \iff \sigma_{n-1-(j-i)}^{p(\XX)} (x,y) \\
(\forall i\leq j) (\forall x\in \mathcal{P}_i^{\XX}) (\forall y\in \mathcal{P}_j^{\XX})\; y<^{\XX}x \iff \sigma_{j-i}^{p(\XX)} (x,y)
\end{align*}
In the case of $n=2$, our $p$ coincides exactly with the one (implicitly) defined in \cite{ColouredMilliken}.  I claim that this map is an expansion. 
\begin{prop}
	The map $p: \text{Age}(\QQ_n) \rightarrow \text{Age}(\SSSS(n))$ is a forgetful functor that witnesses $\Age(\QQ_n)$ is a categorical precompact expansion of $\Age(\SSSS(n))$.
\end{prop}
\begin{proof}
	We need to show $p$ satisfies (Ref), (Proj) and (C). For any $\XX, \YY \in \text{Age}(\QQ_n)$, $p:\text{hom}(\XX,\YY) \rightarrow   \text{hom} (p(\XX), p(\YY))$ is injective. Take $f:
	\XX\rightarrow \YY$ an embedding. So, $\forall x,y \in X$ if $x<^{\XX} y$ then $f(x)<^{\YY} f(y)$. Similarly, if $x\in \mathcal{P}_i^{\XX}$, then $f(x) \in \mathcal{P}_i^{\YY}$. Consequently, by how we defined $p$, $\sigma_{i}^{p(\XX)} (x,y) \rightarrow \sigma_i^{p(\YY)} (f(x),f(y))$ which means $f\in \text{hom}(p(\XX), p(\YY))$ and $p$ satisfies (C). \\
	\\
	Now we must show that $p$ is surjective over objects. Take an arbitrary $\AA \in \text{Age}(\SSSS(n)) $. We split $S(n)$ in to $n$ quadrants defined by $Q_k = \{ x\in S(n) : \text{arg}(x) \in (\frac{2\pi (k-1) }{n},\frac{2\pi k }{n} ) \}$ for $k=1,..,n$. Then, we define a partition
	\begin{align*}
	\mathcal{P}_k^{\XX} &= \{e^{-\frac{2\pi i(k-1)}{n}}x : x\in A\cap Q_k  \}
	\end{align*}
	Setting $X= \bigcup_{k=1}^n \mathcal{P}_k^{\XX}$, we see that $X\subseteq Q_1$ and so $\sigma_0^{\SSSS(n)}$ induces a linear order on $X$. The model $\XX = (X,\mathcal{P}_1^{\XX},...,\mathcal{P}_n^{\XX}) $ with order $<^{\XX} = \sigma_0 \upharpoonright{X}$ is an element of $\QQ_n$. Now, we need only check that $p(\XX) = \AA$. To see this, notice that if $x\in \mathcal{P}_k^{\XX}$ and $y\in \mathcal{P}_j^{\XX} $ with $k\leq j$ and $x<^{\XX} y$, then $\text{arg}(\frac{y}{x} ) \in (0, \frac{2\pi}{n}) $. So,
	\begin{align*}
	\text{arg}( \frac{e^{\frac{2\pi i(j-1)}{n}}y }{e^{\frac{2\pi i(k-1)}{n}}x } ) &= \text{arg} (e^{2\pi i \frac{(j-k)}{n}} \frac{y}{x})  \\
	&= \frac{2\pi (j-k)}{n}  + \text{arg}( \frac{y}{x}) \in ( \frac{2\pi (j-k)}{n}, \frac{2\pi (j-k+1)}{n}
	\end{align*}
	So then, in $\AA$, we have $\sigma_{j-k}^{\AA} (e^{\frac{2\pi i(j-1)}{n}}y),e^{\frac{2\pi i(k-1)}{n}}x)$. But, in $p(\XX)$, we have $\sigma_{n-1-(j-k)}^{p(\XX)}(x,y) \iff \sigma_{j-k}^{p(\XX)} (y,x)$. The case for $x\in \mathcal{P}_k^{\XX}$ and $y\in \mathcal{P}_j^{\XX} $ with $k\leq j$ and $y<^{\XX} x $ is near identical. It becomes clear that $p(\XX)$ is isomorphic to $\AA$. Thus, $p$ satisfies (Proj).  \\
	\\
	If $f: \AA \rightarrow \BB $ is an embedding for $\AA,\BB \in \Age(\SSSS(n)) $, then without loss of generality, we may assume $A\subseteq B$. If we do the reversal procedure outlined above to get an expansion $\XX$ of $\BB $ (it will be shown in Proposition 2 that all expansions can be gotten via the reversal procedure) then it is clear that by ignoring the points in $B\setminus A$, we have also done the reversal procedure to $A$ to get an expansion $\YY$ with $f^*:\YY\rightarrow \XX $ an embedding. Moreover, it is clear that the image of $f^*$ supports the image of $f$. Thus, $p$ satisfies (Ref) and hence, is an expansion. 
	\\
	\\
	Precompactness is a trivial consequence of the fact that $p$ sends structures with universe of size $N$ to structures with universe of size $N$. Since there are exactly $n! n^N$ many members of $\Age(\QQ_n)$ with cardinality $N$, any $\AA \in \Age(\SSSS(n))$ has $m(\AA) \leq n! n^N$. 
\end{proof}
\begin{corollary}
	The class $\text{Age}(\SSSS(n))$ has finite Ramsey degrees.
\end{corollary}
\begin{proof}
	Note that $\text{Age}(\QQ_n)$ has the Ramsey property. So, for any $\AA \in \Age(\SSSS(n)) $, $ t_{\Age(\SSSS(n))} (\AA) \leq m(\AA) \leq n! n^{|\AA|} $.
\end{proof}
Note that we have a very crude upper bound. However, using the fact that there are exactly ~$n! n^N$ many $\XX \in \text{Age}(\QQ_n)$ with $|X| = N$, and that expansions of any $\AA \in \text{Age}(\SSSS(n))$ are unique, 
\begin{align*}
\sum\limits_{\substack{\AA \in \text{Age}(\SSSS(n)) \\ |A|= N}} m(\AA) &\leq  n! n^N
\end{align*}
With equality in the case that we can show the condition of theorem 3.2. So indeed, the value of $m(\AA)$ should be much less than $n! n^N $. In fact, claim that we can actually compute the value.
\begin{prop}
	For any $\AA \in \text{Age} (\SSSS(n))  $, $m(\AA) = \frac{n |A|}{|\text{Aut}(\AA)|} $.
\end{prop}
\begin{proof}
	To see this, we use the fact that any expansion of $\AA$ can be gotten from the reversible process outlined in proposition 1. The process has two main steps, first defining a partition of $S(n)$ in to $n$ quadrants $Q_1,...,Q_n$ of the form $Q_k =  \{ x\in S(n) : \text{arg}(x) \in (\frac{2\pi (k-1)  }{n} +\theta ,\frac{2\pi k }{n} +\theta ) \}$ for some $\theta $. Viewing $A$ as a subset of $S(n)$, we need to determine which members of $A$ belong to which quadrants $Q_k$. Note, since we are working with a copy of $\AA$ as it appears in $\SSSS(n)$, we are doing the reversal procedure up to an automorphism of $\AA$ i.e we are counting up to automorphism.\\
	\\
	\textbf{Claim:} There are exactly $|A|$ many unique ways to partition $A$ in this manner.
	\begin{proof}
		We do this by induction on the cardinality of $A$. The case $|A|=1$ is trivial. Suppose it is true for structures with $|A|=k$. Suppose we have a structure $\AA$ with $|A| = k+1$. Choose the $x\in A$ which minimizes $\text{min}_{k} |\text{arg}(\frac{y}{e^{\frac{2\pi i k}{n}  }})| $. While it is clear a member must witness a minima, it's uniqueness comes from the fact that if two members have the same minima, they must differ in argument by a factor of the form $\frac{2\pi k}{n} $ which is a contradiction. Set $x = x_1$ and label the members of $A$ in the order they appear rotating counter clockwise from $x_1$. We now have $A= \{x_1,...,x_{k+1}\} $. Let $\AA^{-}$ be the structure induced by $A^{-} = \{x_2,...,x_{k+1} \}$. We let $Q_j^\theta (B)  = \{ y\in S(n) \cap B : \text{arg}(y) \in (\frac{2\pi (j-1)  }{n} +\theta ,\frac{2\pi j }{n} +\theta )  \} $. By our induction hypothesis, there are $k$ many increasing $\theta_j \in (0,\frac{2\pi i }{n})$ such that $\{ \{ Q_l^{\theta_j} (A^{-})  :l=1,...,n \} : j=1,...,k  \} $ exhausts all unique quadrant partitions of $A^{-}$. It is clear that $\{ \{ Q_l^{\theta_j} (A)  :l=1,...,n \} : j=1,...,k  \}$ has $k$ many members and by the minimality of $x_1$, we may assume that $\theta_j$ were chosen so that there is a unique $l$ for which $x\in Q_l^{\theta_j} (A)$ for all $j$. Ofcourse, also by minimality, we can choose $\epsilon $ larger than the minimal distance from $x$ to a border of a $Q_k$ yet smaller than any other distance, so that $ \{ Q_l^{\theta_k+\epsilon} (A)  :l=1,...,n \} \notin \{ \{ Q_l^{\theta_j} (A)  :l=1,...,n \} : j=1,...,k  \}$.  Notice that this collection is the exact same as $\{Q_j \cap A : j=1,...,n  \}$, so by the cyclic nature of $Q_l^{\theta} (A)$, we have exactly $k+1$ many partitions of $A$. Namely, $\{\{ Q_l^{\theta_k+\epsilon} (A)  :l=1,...,n \} \} \cup   \{ \{ Q_l^{\theta_j} (A)  :l=1,...,n \} : j=1,...,k  \} $. See figure 1 for an example of some quadrant placements
		\begin{figure}
			\centering
			
			\begin{tikzpicture}
			\draw  circle (1cm); 
			\node[circle, fill=black,scale=0.25] (x_1) at (5 :1cm) {$x_1$ };
			\node[circle, fill=black,scale=0.25] (x_2) at (100 :1cm) {$x_2$};
			\node[circle, fill=black,scale=0.25] (x_3) at (200 :1cm) {$x_3$};
			\node[circle, fill=black,scale=0.25] (x_4) at (300 :1cm) {$x_4$};
			\node[circle, fill=black, scale=0.1] (east) at (0 :1cm) {}; 
			\node[circle, fill=black, scale=0.1] (north) at (90 :1cm) {};
			\node[circle, fill=black, scale=0.1] (west) at (180 :1cm) {};   
			\node[circle, fill=black, scale=0.1] (south) at (270 :1cm) {};

			\draw[black] (north) edge (south);
			\draw[black] (west) edge (east);
			
			\end{tikzpicture}
			\begin{tikzpicture}
			\draw  circle (1cm); 
			\node[circle, fill=black,scale=0.25] (x_1) at (5 :1cm) {$x_1$ };
			\node[circle, fill=black,scale=0.25] (x_2) at (100 :1cm) {$x_2$};
			\node[circle, fill=black,scale=0.25] (x_3) at (200 :1cm) {$x_3$};
			\node[circle, fill=black,scale=0.25] (x_4) at (300 :1cm) {$x_4$};
			\node[circle, fill=black, scale=0.1] (east) at (8 :1cm) {}; 
			\node[circle, fill=black, scale=0.1] (north) at (98 :1cm) {};
			\node[circle, fill=black, scale=0.1] (west) at (188 :1cm) {};   
			\node[circle, fill=black, scale=0.1] (south) at (278 :1cm) {};

			\draw[black] (north) edge (south);
			\draw[black] (west) edge (east);
			
			\end{tikzpicture}
			\begin{tikzpicture}
			\draw  circle (1cm); 
			\node[circle, fill=black,scale=0.25] (x_1) at (5 :1cm) {$x_1$ };
			\node[circle, fill=black,scale=0.25] (x_2) at (100 :1cm) {$x_2$};
			\node[circle, fill=black,scale=0.25] (x_3) at (200 :1cm) {$x_3$};
			\node[circle, fill=black,scale=0.25] (x_4) at (300 :1cm) {$x_4$};
			\node[circle, fill=black, scale=0.1] (east) at (14 :1cm) {}; 
			\node[circle, fill=black, scale=0.1] (north) at (104 :1cm) {};
			\node[circle, fill=black, scale=0.1] (west) at (194 :1cm) {};   
			\node[circle, fill=black, scale=0.1] (south) at (284 :1cm) {};

			\draw[black] (north) edge (south);
			\draw[black] (west) edge (east);
			
			\end{tikzpicture}
			\begin{tikzpicture}
			\draw  circle (1cm); 
			\node[circle, fill=black,scale=0.25] (x_1) at (5 :1cm) {$x_1$ };
			\node[circle, fill=black,scale=0.25] (x_2) at (100 :1cm) {$x_2$};
			\node[circle, fill=black,scale=0.25] (x_3) at (200 :1cm) {$x_3$};
			\node[circle, fill=black,scale=0.25] (x_4) at (300 :1cm) {$x_4$};
			\node[circle, fill=black, scale=0.1] (east) at (25 :1cm) {}; 
			\node[circle, fill=black, scale=0.1] (north) at (115 :1cm) {};
			\node[circle, fill=black, scale=0.1] (west) at (205 :1cm) {};   
			\node[circle, fill=black, scale=0.1] (south) at (295 :1cm) {};

			\draw[black] (north) edge (south);
			\draw[black] (west) edge (east);
			
			\end{tikzpicture}
			\caption{All possible quadrant placements for a size four member of $\Age(\SSSS(4)) $.}
			\label{fig:my_label}
		\end{figure}
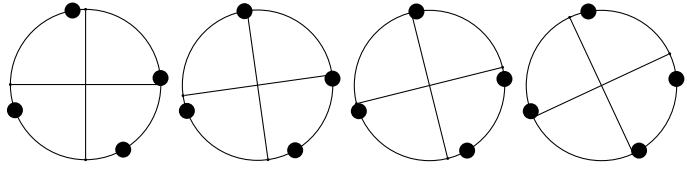
		
	\end{proof}
	After we do this, we decide which member of the partition we will rotate all the other points to and then let the linear order be defined by $\sigma_1^{\AA}$. There are exactly $n$ quadrants. Given that making any different choice in step one and two, would necessarily lead to a different linear order under $\sigma_1^{\AA}$, we have exactly $n|A|$ unique linear orders up to automorphism of $\AA$. Or rather, $m(\AA) = \frac{n|A|}{|\text{Aut}(\AA)|} $. 
\end{proof}
If we want an exact computation of Ramsey degrees, we still need to use theorem 3.2. We show the remaining fact now. 
\begin{theorem}
	$\forall \AA \in \Age(\SSSS(n))$, $t_{\Age(\SSSS(n))}  (\AA) = \frac{n|A|}{|\Aut(\AA)|}$. 
\end{theorem}
\begin{proof}
	Take an arbitrary $\AA \in \Age(\SSSS(n))$. We need to find a $\BB \in \Age(\SSSS(n))$ such that any expansion $\AA^i$ of $\AA$ in $\Age(\QQ_n)$, embeds in to any expansion $\BB^j$ of $\BB$ in $\Age(\QQ_n)$. Consider the structures $\CC_m $ in $\SSSS(n)$ induced by the points $\{ e^{ \frac{2k\pi i}{nm+1}  } :  k=0,...,nm  \}  $. Since $C_m$ has at least $m$ points in each region $Q_k = \{ y: \text{arg}(y) \in (\frac{2\pi i k}{n} , \frac{2\pi i (k+1)}{n})  \}$, and since $\bigcup\limits_{m=1}^\infty\{ e^{ \frac{2k\pi i}{nm+1}  } :  k=0,...,nm  \}  $ is dense in the unit circle, there is always an $m$ large enough so that $\AA$ embeds into $C_m$ and $m>|A|$. Note that rotations of the form $e^{ \frac{2\pi i k}{nm+1} } $ are all automorphisms of $\CC_m$, so $|\Aut(\CC_m)| \geq nm+1 $. Moreover, we know of $n$ expansions of $\CC_m$. Namely, partitioned linear orders of the form $\{ (\{1,...,nm+1\},< ): \mathcal{P}_1^j,...,\mathcal{P}_n^j   \}$, j=1,...,n, where $<$ is the standard order and the partitions are
	\begin{align*}
	\mathcal{P}_k^j = \{ x \in \{1,...,nm+1 \} :  x \equiv  k+j \text{ mod }n  \}
	\end{align*}
	So, $m(\CC_m) \geq n$. But by Proposition $2$, this means that $\frac{n (nm+1) }{ \Aut(\CC_m)} \geq n \Rightarrow \Aut(\CC_m) \leq nm+1$. So, $\Aut(\CC_m) = nm+1$ and thus $m(\AA) = n$ which means our above expansions are actually an exhaustive list. Finally, since $m> |A|$, and each $\mathcal{P}_k^j$ has at least $m$ elements, any n-partitioned linear order embeds into $\{ (\{1,...,nm+1\},< ): \mathcal{P}_1^j,...,\mathcal{P}_n^j   \}$ for all $j$. We will show this briefly. After doing so, we are done.\\
	\\
	Suppose $\XX$ is an extension of $\AA$. So, $\XX = \{(\{x_1,...,x_{|A|} \} , \prec) U_1,.., U_n   \} $ where $x_a \prec x_b \iff  a<b $. We define a map $ f: \XX \rightarrow \{ (\{1,...,nm+1\},< ): \mathcal{P}_1^j,...,\mathcal{P}_n^j   \} $ as follows. If $x_a \in U_b$, then $f(x_a)$ will be the $a$th member of $\mathcal{P}_a^j $. This is well defined as each $\mathcal{P}_k^j$ has at least $m$ members and $m> |A|$. Moreover, it is clear that $f$ respects partitions i.e it sends members from the $b$th piece to members of the $b$th piece. However, since it is always the case that if $a<b$ then the $a$th member of $ \mathcal{P}_k^j$ is always less than the $b$th member of $\mathcal{P}_c^j $ (regardless of $k$ and $j$), $f$ also respects $\prec$. That is, if $x_a \prec x_b$, then $f(x_a) < f(x_b)$. Therefore, $f$ is an embedding. As this held for any $j$, any expansion of $\XX$ embeds in to any expansion of $\CC_m$ (for appropriately chosen $m$). 
\end{proof}
\begin{corollary}
	We have the following equality. 
	\begin{align*}
	\frac{N}{(n-1)!}\sum\limits_{\substack{\AA \in \text{Age}(\SSSS(n)) \\ |A|= N}} \frac{1}{|\text{Aut}(\AA)|} &=   n^N
	\end{align*}
\end{corollary}
We now conclude with a result about big Ramsey degrees.
\begin{theorem}
	For any $\AA \in \text{Age}(\SSSS(n))$, $T_{\text{Age}(\SSSS(n))}(\AA) = m(\AA)\tan^{(2|A|-1)}(0)  $.
\end{theorem}
\begin{proof}
	It is clear that the big Ramsey degree of any $\XX \in \text{Age}(\QQ_n)$ is $\tan^{(2|X|-1)}(0)$, so it suffices to show that extending $p$ to $ \sigma \text{Age}(\QQ_n)$ will ensure $p$ satisfies the conditions of theorem 4.3. It is clear that the way we defined $p$ did not depend on the input set $\XX$ being finite. For this reason, extending $p$ to a map from $\overline{\text{Age}(\QQ_n)}$ to $\overline{Age (\SSSS(n))} $ is trivial. We just define $p$ as we have before on countable structures, and this will match with how $p$ ought to be extended as seen in section 3. The more tricky thing to show is that $\QQ_n$ is the only expansion of $\SSSS(n)$. \\
	\\
	It is sufficient to show that doing the reversal procedure (outlined in proposition 2) to $\SSSS(n)$ will grant us a model isomorphic to $\QQ_n$.\\
	\\
	Split $S(n)$ in to $n$ quadrants defined by $Q_k = \{ x\in S(n) : \text{arg}(x) \in (\frac{2\pi (k-1) }{n},\frac{2\pi k }{n} ) \}$ for $k=1,..,n$. Define a partition
	\begin{align*}
	\mathcal{P}_k^{\XX} &= \{e^{-\frac{2\pi i(k-1)}{n}}x : x\in  Q_k  \}
	\end{align*}
	The model $\XX = (X,\mathcal{P}_1^{\XX},...,\mathcal{P}_n^{\XX}, <^{\XX}) $ where $x<^{\XX}y$ if and only if $\text{arg}(\frac{x}{y}) \in (0, \frac{2\pi}{n})$. So, $X$ paired with $<^{\XX}$ can be identified with a countable dense subset of $(0,\frac{2\pi}{n}) $, and hence is isomorphic to $\QQ$. It suffices now to show that $<^{\XX} \upharpoonright{\PP_k^{\XX}}$ is also a dense linear order without endpoints. This is not hard to show as $Q_k $ is a dense linear order without endpoints, where our order is defined by $\sigma_0^{\SSSS(n)}$. Consider the bijection $f:Q_k \rightarrow \mathcal{P}_k^{\XX} $ defined by $x\rightarrow e^{-\frac{2\pi i(k-1)}{n} }x$. Since $g: \SSSS^1\times \SSSS^1 \rightarrow \SSSS^1$ defined by $g(x,y) =\text{arg}(\frac{x}{y})$ is invariant under shifts, i.e $g(x,y) = g(e^{i\theta}x, e^{i\theta}y)$, we have
	\begin{align*}
	\sigma_0^{\SSSS(n)}(x,y) &\iff \text{arg}(\frac{x}{y}) \in (0,\frac{2\pi}{n})\\
	&\iff \text{arg}(\frac{f(x)}{f(y)}) \in (0,\frac{2\pi}{n})\\
	&\iff f(x) <^{\XX} f(y) 
	\end{align*}
	So, $f$ is an isomorphism between linear orders. Since $(Q_k, \sigma_0^{\SSSS(n)})$ is isomorphic to $\QQ$, so is each $\mathcal{P}_k^{\XX}$ with respect to $<^\XX$. Consequently, $ \XX$ is isomorphic to $\QQ_n$, meaning any expansion of $\SSSS(n)$ is isomorphic to $\QQ_n$ and so, ${{\QQ_n}\choose{\QQ_n}} = {{\SSSS(n)}\choose{\SSSS(n)}}$. 
\end{proof}
\section{Acknowledgements}
The author would like to thank Lionel Nguyen van Th\'e, Natasha Dobrinen and Wieslaw Kubis for the insightful comments they gave at the 50 years of Set Theory in Toronto conference. The author would also like to thank his supervisor Stevo Todorcevic for his guidance. 


\begin{thebibliography}{9}
	\bibitem{KPT} 
	A. S. Kechris, V. G. Pestov, and S. Todorcevic. 
	\textit{ Fra{\"i}ss{\'e} limits, Ramsey theory, and
		topological dynamics of automorphism groups}. 
	Geom. Funct. Anal. \textbf{15} (2005), no. 1,
	106–189.
	
	\bibitem{Kubis}
	W. Kubi\'s. \textit{Fra\"iss\'e sequences: category-theoretic approach to universal homogeneous structures}. 
	Annals of Pure and Applied Logic, Volume 165, Issue 11, November 2014, Pages 1755-1811
	\bibitem{ColouredMilliken}
	C. Laflamme, L. Nguyen Van Th{\'e}, N. W. Sauer. \textit{Partition properties of the dense local order and a colored version of Milliken's theorem}. 
	Combinatorica, 30 (1), 83-104, 2010. 	
	
	\bibitem{BigRamsey}
	D. Masulovic. \textit{Finite big Ramsey degrees in universal structures}. 
	\\\texttt{ arXiv:1807.00658v2}
	
	\bibitem{PreadjRamsey}
	D. Masulovic. \textit{Pre-adjunctions and the Ramsey property
	}. 
	\\\texttt{arXiv:1609.06832v3}
	
	\bibitem{HomogSurvey} 
	H.D Macpherson. 
	\textit{A Survey on Homogeneous Structures}. 
	\\\texttt{http://ambio1.leeds.ac.uk/Pure/staff/macpherson/homog7.pdf}
	
	\bibitem{MetUniv} 
	J. Melleray, L. Nguyen van Th{\'e}, T. Tsankov.
	\textit{Polish groups with metrizable universal minimal flows}.
		International Mathematics Research Notices. IMRN, 5, 1285-1307, 2016 
	
	\bibitem{PrecompExp} 
	L. Nguyen van Th{\'e}. 
	\textit{More on the Kechris-Pestov-Todorcevic correspondence: precompact expansions}. 
	Fund. Math., 222, 19-47, 2013. 
	
	\bibitem{SokicPoset}
	M. Sokic. \textit{Ramsey Property of Posets and Related Structures
	} . Thesis, University of Toronto, 2011. 
	
	\bibitem{Zucker}
	A. Zucker. \textit{Amenability and Unique Ergodicity of Automorphism Groups of Fraïssé Structures}.
	arXiv:1304.2839
\end{thebibliography}
\end{document}